\documentclass[11pt, oneside, a4paper]{article}
\usepackage[utf8]{inputenc}
\usepackage[T1]{fontenc}
\usepackage{amsmath}
\usepackage{amsthm,enumerate} 
\usepackage{amssymb}
\usepackage{listings}
\usepackage{url}
\usepackage{hyperref}

\newtheorem{thm}{Theorem}[section]
\newtheorem{lemma}[thm]{Lemma}
\newtheorem{theorem}[thm]{Theorem}
\newtheorem{cor}[thm]{Corollary}
\newtheorem{prop}[thm]{Proposition}
\newtheorem{notation}[thm]{Notation}
\newtheorem{remark}[thm]{Remark}
\newtheorem{l+d}[thm]{Lemma und Definition}
\newtheorem{k+d}[thm]{Corollary und Definition}

\newtheorem{example}[thm]{Example}

\newtheorem{bem+de}[thm]{Defintion+Bemerkung}

\newcommand{\erz}[1]{\langle #1\rangle}
\DeclareMathOperator{\PG}{PG}

\title{Maximal cocliques and the chromatic number of the Kneser graph on
chambers of PG$(3,q)$}

\author{Philipp Heering, Klaus Metsch}
\date{November 17, 2023}

\begin{document}

\maketitle

\begin{center}
\section*{Abstract}
\end{center}

Let $\Gamma$ be the graph whose vertices are the chambers of the finite projective $3$-space $\PG(3,q)$, with two vertices being adjacent if and only if the corresponding chambers are in general position. We show that a maximal independent set of vertices of $\Gamma$ contains $q^4+3q^3+4q^2+3q+1$, or $3q^3+5q^2+3q+1$, or at most $3q^3+4q^2+3q+2$ elements. For $q\geq 4$ the structure of the largest maximal independent sets is described. For $q\geq 7$ the structure of the maximal independent sets of the three largest cardinalities is described. Using the cardinality of the second largest maximal independent sets, we show that the chromatic number of $\Gamma$ is $q^2+q$. 

\textbf{Keywords:} $q$-analog of generalized Kneser graph, maximal independent set, chromatic number\\
\textbf{MSC (2020):} 
05C69, 
51E20, 
05C15 

\section{Introduction}

We consider \emph{chambers} $\{P,\ell,\pi\}$ of a finite projective $3$-space consisting of a point $P$, a line $\ell$ and a plane $\pi$ that are mutually incident. Two chambers $\{ P_1,l_1,\pi_1\}$ and $\{ P_2,l_2,\pi_2\}$ are called \emph{opposite} if $P_1$ does not lie in $\pi_2$, if $P_2$ does not lie in $\pi_1$ and if $\ell_1$ and $\ell_2$ are skew. Let $\Gamma$ be the graph whose vertices are the chambers of $\PG(3,q)$ where two vertices are adjacent if and only if the corresponding chambers are opposite. We call this graph the \emph{Kneser graph on chambers  of $\PG(3,q)$}. Independent sets of this graph are sets of chambers that are pairwise non-opposite. We study maximal independent sets of $\Gamma$.
Similar problems have been studied extensively, see for example \cite{EKRpointplaneflags4d},  \cite{cocliquesonlineplanein4d}, \cite{pointhyperplaneflags},  
	\cite{AlgebraicApproach}, 
	\cite{Thechromaticnumberoftwo}.

The independence number of $\Gamma$ is $(q^2+q+1)(q+1)^2$, as was shown in Theorem 3.1 in \cite{AlgebraicApproach} with algebraic methods. Let $x$ be a point or a plane in $\PG(3,q)$. The set of all chambers, whose line is incident with $x$, is denoted by $C(x)$. Since the lines of the chambers in this set mutually meet, $C(x)$ is an independent set and it is easy to see that it is in fact a maximal independent set. Furthermore $|C(x)|=(q^2+q+1)(q+1)^2$ for all points and planes $x$.

With a geometric approach, it was shown in \cite{Thechromaticnumberoftwo} that the second largest maximal independent sets have at most $45q^3+45q^2+9q+9 $ elements. We also use a purely geometrical approach. We reprove the bound for the independence number found in \cite{AlgebraicApproach} and find a tight bound for the cardinality of the second and third largest maximal independent sets of $\Gamma$. Our main result is the following.

\begin{theorem}\label{T: main theorem}
Let $M$ be a maximal independent set of the Kneser graph on chambers  of $\PG(3,q)$. 
\begin{itemize}
	\item[(a)] If $M$ has at least $3q^3+4q^2+3q+2$ elements, then $M$ contains exactly $q^4+3q^3+4q^2+3q+1$, or $3q^3+5q^2+3q+1$, or $3q^3+4q^2+3q+2$ chambers.
	\item[(b)]  If $q\geq 7$ and $|M|\geq 3q^3+4q^2+3q+2$, then $M$ is one of the independent sets described in Section \ref{examplesection}, or dual to one such set.
	\item[(c)] If $q\geq 4$ and $|M|=q^4+3q^3+4q^2+3q+1$, then $M$ is one of the sets described in Example \ref{E: largest structure}.
\end{itemize}

\end{theorem}

One of the main observations for the proof of Theorem \ref{T: main theorem} is the following. If $M$ is a maximal independent set, then every line of $\PG(3,q)$ occurs in exactly no, one, two, $q+1$, $2q+1$, or $(q+1)^2$ chambers of $M$. This is discussed in Section \ref{Section: Preliminary}. A somewhat similar observation is used in \cite{cocliquesonlineplanein4d}. For $q\leq 5$ we  calculate the cardinality of the second and third largest maximal independent sets using the computer, this is used for parts of (a) and (c) of Theorem \ref{T: main theorem}.

Similar to Theorem 18 in \cite{Thechromaticnumberoftwo}, we use the cardinality of the second largest maximal sets and the structure of the largest maximal independent sets to determine the chromatic number of $\Gamma$, in \cite{Thechromaticnumberoftwo} this was done only for $q\geq 47$. 

\begin{theorem} \label{T: chromatic number}
	The chromatic number of the Kneser graph on chambers  of $\PG(3,q)$ is $q^2+q$.
\end{theorem}

	For $q=2,3$ the chromatic number of the Kneser graph on chambers  of $\PG(3,q)$ is calculated by the computer. Theorem \ref{T: main theorem} is used in the proof of Theorem \ref{T: chromatic number}, for $q=4,5$ we also use parts of Theorem \ref{T: main theorem} that were shown with the help of the computer.


\section{Preliminary results} \label{Section: Preliminary}

 In this section we introduce some notation and prove results that hold for any maximal independent set of the Kneser graph on chambers  of $\PG(3,q)$. This properties will play a prominent role in the characterization of large maximal independent sets in Section \ref{Section proof of main theorem}.

\begin{notation}
\begin{enumerate}[\rm 1.]
\item Given lines $\ell$ and $h$ we mean by $\ell$ \emph{meets} $h$ that the lines meet non-trivially, that is, they are equal or share exactly one point.
\item A chamber $\{P,\ell,\pi\}$ with a point $P$, a line $\ell$, and a plane $\pi$ is denoted as a triple $(P,\ell,\pi)$.
\item Chambers of an independent set $M$ are also called \emph{$M$-chambers}.
\item Lines of chambers of an independent set $M$ are also called \emph{$M$-lines}. We also write \emph{lines in $M$}, instead of $M$-lines.
\item  If $\ell$ is a line, then the number of chambers of an independent set $M$ that contain $\ell$ is called the $M$-\emph{weight} of $\ell$ or, if there is no risk of confusion, the \emph{weight} of $\ell$.
\item If $(\ell,\pi)$ is an incident line-plane pair, then the number of chambers of an independent set $M$ that contain $\ell$ and $\pi$ is called the $M$-\emph{weight}  or just the \emph{weight} of $(\ell,\pi)$. If $(P,l)$ is an incident point-line pair, then the number of chambers of an independent set $M$ that contain $P$ and $\ell$ is called the $M$-\emph{weight} or \emph{weight} of $(P,l)$.
\end{enumerate}
\end{notation}

\begin{lemma} \label{L: weight of pairs}
Let $M$ be a maximal independent set of chambers of $\PG(3,q)$.
\begin{enumerate}[\rm (a)]
\item 	Every incident line-plane pair $(\ell,\pi)$ has weight $0$, $1$, or $q+1$. If 
		it has weight $q+1$, every $M$-chamber $(Q,h,\tau)$ satisfies $h\cap \ell \neq \emptyset$, or $Q\in \pi$.
\item Every incident point-line pair $(P,\ell)$ has weight $0$, $1$, or $q+1$. If 
		it has weight $q+1$, every $M$-chamber $(Q,h,\tau)$ satisfies $h\cap \ell \neq \emptyset$, or $P\in \tau$.
\end{enumerate}
\end{lemma}

\begin{proof}
The statements (a) and (b) are dual, therefore it suffices to prove the first claim. To see this suppose that $(\ell,\pi)$ is an incident line-plane pair that has weight at least two, so that $M$ contains chambers $(P_1,\ell,\pi)$ and $(P_2,\ell,\pi)$ with distinct points $P_1$ and $P_2$ of $\ell$. Let $P$ be any point of $\ell$. We shall show that the chamber $(P,\ell,\pi)$ is not opposite to any chamber of $M$.
 Then maximality of $M$ implies that $(P,\ell,\pi)\in M$ and since this applies to all points $P$ of $\ell$, it follows that $(\ell,\pi)$ has weight $q+1$.

Let $(Q,h,\tau)$ be any chamber of $M$. If $h$ is not skew to $\ell$, then $(Q,h,\tau)$ and $(P,\ell,\pi)$ are not opposite. So suppose that $h$ and $\ell$ are skew. Then $\tau$ does not contain $\ell$, so w.l.o.g. $P_1\notin\tau$. Since $(P_1,\ell,\pi)$ and $(Q,h,\tau)$ are not opposite, this implies that $Q\in\pi$. Hence $(Q,h,\tau)$ and $(P,\ell,\pi)$ are not opposite.
\end{proof}

\begin{prop}  \label{P: weight of lines}
Let $M$ be a maximal independent set of chambers of $\PG(3,q)$ and let $\ell$ be a line.
Then $\ell$ has weight $0$, $1$, $2$, $q+1$, $2q+1$ or $(q+1)^2$. Moreover, the following hold.
\begin{enumerate}[\rm (a)]
\item The line $\ell$ has weight $(q+1)^2$ if and only if it meets the line of every $M$-chamber.
\item If $\ell$ has weight $2q+1$, then $\ell$ is incident with a point $P$ and a plane $\pi$ such that the $M$-chambers that contain $\ell$ are the chambers that contain $\ell$ and $P$ or $\pi$. Furthermore every $M$-chamber $(Q,h,\tau)$ with $h\cap\ell\not=\emptyset$ satisfies $Q\in\pi$ and $\tau\ni P$. In particular, every $M$-line $h$ with $h\cap\ell=\emptyset$ has weight $1$.
\item If $\ell$ has weight $q+1$, then one of the following two cases occurs.
\begin{enumerate}[(i)]
\item There exists a plane $\pi$ on $\ell$, such that the $M$-chambers of $\ell$ are the $q+1$ chambers that contain $\pi$ and $\ell$. In this case every $M$-chamber $(Q,h,\tau)$ satisfies $h\cap\ell\not=\emptyset$ or $Q\in\pi$.
\item There exists a point $P$ on $\ell$, such that the $M$-chambers of $\ell$ are the $q+1$ chambers that contain $P$ and $\ell$. In this case every $M$-chamber $(Q,h,\tau)$ satisfies $h\cap\ell\not=\emptyset$ or $P\in\tau$.
    \end{enumerate}
\item If $\ell$ has weight two and $(P_i,\ell,\pi_i)$, $i=1,2$ are the two chambers of $M$ containing $\ell$, then $P_1\not=P_2$ and $\pi_1\not=\pi_2$. In this case, every chamber $(Q,h,\tau)$ in $M$ with a line $h$ that is skew to $\ell$ satisfies either $P_1\in \tau$ and $Q\in\pi_2$ or $P_2\in\tau$ and $Q\in \pi_1$.
\end{enumerate}
\end{prop}
\begin{proof}
If all $M$-lines meet $\ell$, then each of the $(q+1)^2$ chambers on $\ell$ is non-opposite to any $M$-chamber, so that maximality of $M$ implies that these $(q+1)^2$ chambers all belong to $M$, so $\ell$ has weight $(q+1)^2$.

From now on we assume $M$ contains a flag $(Q_0,h_0,\tau_0)$ with $h_0$ skew to $\ell$. As $h_0$ is skew to $\ell$, then $\tau_0$ meets $\ell$ in a point $P_0$, and $\ell$ and $Q_0$ generate a plane $\pi_0$. If $(P,\ell,\pi)$ is any chamber of $M$ containing $\ell$, then it is non-opposite to $(Q_0,h_0,\tau_0)$ and hence $P\in\tau_0$ or $Q_0\in\pi$, that is $P=P_0$ or $\pi=\pi_0$. Lemma \ref{L: weight of pairs} implies that the number of chambers of $M$ containing $P_0$ and $\ell$ is either $0$, $1$ or $q+1$, and that the number of chambers of $M$ containing $\pi_0$ and $\ell$ is either $0$, $1$ or $q+1$. If $q+1$ occurs in any of the two situations mentioned above, then $(P_0,\ell,\pi_0)$ lies in $M$. It follows immediately that the number of $M$-chambers on $\ell$ is $0$, $1$, $2$, $q+1$ or $2q+1$. It remains to prove (b)-(d).

(b) We assume that $\ell$ lies in exactly $2q+1$ $M$-chambers. Then these are all chambers containing $P_0$ and $\ell$ or containing $\ell$ and $\pi_0$.
	For the second statement in (b) suppose that $(Q,h,\tau)$ is an $M$-chamber with $h$ skew to $\ell$. Since $(Q,h,\tau)$ is non-opposite to the chambers containing $P_0$ and $\ell$, we have $P_0\in \tau$. Dually we get $Q\in \pi_0$. This proves (b).

(c) We assume that $\ell$ has weight $q+1$. Then these are either the $q+1$ chambers containing $P_0$ and $\ell$ or the $q+1$ chambers containing $\pi_0$ and $\ell$.
Suppose first that these are the $q+1$ chambers containing $P_0$ and $\ell$. 
Then Lemma \ref{L: weight of pairs} proves the second part of (c)(i).
Suppose now that the $M$-chambers that contain $\ell$ are the $q+1$ chambers containing $\ell$ and $\pi_0$. This is dual to the previous situation and therefore we are in situation (c)(ii).

(d) It is impossible that $P_1=P_2$, since otherwise Lemma \ref{L: weight of pairs} would imply that $(P_1,\ell)$ has weight $q+1$ and hence $\ell$ would have weight at least $q+1$. Dually we have $\pi_1\not=\pi_2$. For the second statement of (d) suppose that $(Q,h,\tau)$ is a chamber of $M$ with $h$ skew to $\ell$. Since $(P_i,\ell,\pi_i)$ and $(Q,h,\tau)$ are non-opposite, we have $Q\in\pi_i$ or $P_i\in\tau$, for $i=1,2$. Since $h$ lies in $\tau$ and is skew to $\ell$, the plane $\tau$ does not contain $\ell$ and hence $\tau$ contains at most one of the points $P_1$ and $P_2$. As $\pi_1\not=\pi_2$, then $Q$ lies in at most one of these planes. It follows that either $P_1\in \tau$ and $Q\in \pi_2$, or $P_2\in \tau$ and $Q\in \pi_1$.
\end{proof}

\begin{lemma} \label{L: number of lines with weight $(q+1)^2$}
Let $M$ be a maximal  independent set of chambers of $\PG(3,q)$. Then one of the following cases occurs.
\begin{enumerate}[(a)]
\item At most one line has $M$-weight $(q+1)^2$.
\item The number of lines of $M$-weight $(q+1)^2$ is $q+1$ and these lines form a pencil.
\item The number of lines of $M$-weight $(q+1)^2$ is $q^2+q+1$ and there is a point or plane $x$ that is incident with all lines of weight $(q+1)^2$.
    \end{enumerate}
\end{lemma}
\begin{proof}
We use Proposition \ref{P: weight of lines} (a) that states that every line of weight $(q+1)^2$ meets every $M$-line. In particular, the lines of weight $(q+1)^2$ mutually meet. We may assume that there exists at least two lines $\ell_1$ and $\ell_2$ of weight $(q+1)^2$. These meet in a point $P$ and span a plane $\pi$. As every $M$-line meets $\ell_1$ and $\ell_2$, it follows that every $M$-line is incident with $P$ or $\pi$. Therefore the $q+1$ lines that are incident with $P$ and $\pi$, meet all $M$-lines, therefore they have weight $(q+1)^2$. Hence there are at least $q+1$ lines of weight $(q+1)^2$ with equality only if these lines form a pencil.

Now consider the situation when there are more than $q+1$ lines of weight $(q+1)^2$. As they mutually meet there exists a point or plane $x$ incident with all of them. Since there are more than $q+1$ such lines, we see that every line that meets all of them, is also incident with $x$. As lines of weight $(q+1)^2$ meet all $M$-lines, it follows that all $M$-lines are incident with $x$. Then all lines incident with $x$ meet all $M$-lines, so by Proposition \ref{P: weight of lines} (a), all lines incident with $x$ have weight $(q+1)^2$. 
\end{proof}

\begin{lemma} \label{L: s(q^2+q+1) no s+1 skew lines}
Let $n$ be a positive integer. Then every set of lines of $\PG(3,q)$ that does not contain $n+1$ mutually skew lines has cardinality at most $n(q^2+q+1)$.
\end{lemma}
\begin{proof}
Let $S$ be the set of all regular spreads and let $ L$ be the set of all lines of $\PG(3,q)$. We define $T:=\{ (g,F) \in L\times S \mid g \in  F \}$. Since every line lies in the same number $r$ of regular spreads we have $|T|=|L|r$.
On the other hand every spread contains $q^2+1$ lines, hence $|T|=|S|(q^2+1)$. Therefore we have $|L|r=|S|(q^2+1)$.
This yields $|S|=(q^2+q+1)r$.

Let $G$ be a set of lines of $\PG(3,q)$ without $n+1$ mutually skew lines. Then $|F\cap G|\leq n$ for  all $F\in S$. Hence with $T_G:=\{ (g,F)\in  G\times S \mid g\in F \}$ we have $|T_G|\le |S|n$. On the other hand we have $|T_G|=|G|r$. This yields $|G|r\leq |S|n$ and therefore we obtain $|G|\leq n(q^2+q+1)$.
\end{proof}

The following two lemmas will play a crucial role in the classification of large independent sets.

\begin{lemma} \label{L: points in a plane}
Let $M$ be a maximal  independent set of chambers of $\PG(3,q)$. Let $\pi$ be a plane and let $X$ be the set of all chambers $(P,g,\tau)$ of $M$ for which $g$ has $M$-weight $1$ or $2$ and for which $g\not\subseteq\pi$ and $P=g\cap\pi$. Then we have.
\begin{enumerate}[\rm (a)]
\item If the points of the chambers of $X$ are non-collinear, then $|X|\le 3q^2+8q$.
\item If the points of the chambers of $X$ are collinear, then $|X|\le (q+1)q^2$.
\end{enumerate}\end{lemma}
\begin{proof}
If a line has weight two, then by Proposition \ref{P: weight of lines} the two chambers of $M$ on it have distinct points and hence at most one of these two chamber lies in $X$. Therefore distinct chambers of $X$ have distinct lines. In other words, the number of lines that occur in a chamber of $X$ is equal to $|X|$. We call a point, a line or a plane an $X$-point, $X$-line, or $X$-plane if it occurs in a chamber of $X$. Recall that all $X$-points lie in $\pi$ but no $X$-line lies in $\pi$.

If all $X$-points are collinear, then we are in situation (b). In this case there are at most $q+1$ distinct $X$-points and each point is incident with $q^2$ lines that are not in $\pi$, which implies $|X|\leq (q+1)q^2$. Now we consider the situation in which the $X$-points are not collinear. If there do not exist four mutually skew $X$-lines, then Lemma \ref{L: s(q^2+q+1) no s+1 skew lines} yields that $X$ contains at most $3(q^2+q+1)$ lines and we are done. We may thus assume that $X$ contains four chambers $(P_i,g_i,\tau_i)$, $i=1,2,3,4$, whose lines are mutually skew. Then $P_1,\dots,P_4$ are four distinct points.
We distinguish two cases.

\underline{Case 1.} The points $P_1,\ldots,P_4$ are not collinear.

The chambers $(P_i,g_i,\tau_i)$, $i=1,2,3,4$, are mutually not opposite. As $g_i$ and $g_j$ are skew, we must thus have $P_i\in\tau_j$ or $P_j\in\tau_i$ for $1\le i<j\le 4$. Without loss of generality we assume that $P_1$, $P_2$ and $P_3$ are not collinear and that $P_1\in\tau_2$. Then $P_1$ and $P_2$ lie on the line $\tau_2\cap\pi$ and hence $P_3$ does not. Then $P_3\notin\tau_2$ and hence $P_2\in\tau_3$. The same argument shows that $P_3\in\tau_1$. As $P_i\in\tau_i$ for all $i$, it follows that $\tau_1\cap\pi=P_1P_3$, $\tau_2\cap\pi=P_2P_1$, and $\tau_3\cap\pi=P_3P_2$. As $P_4\not=P_1,P_2,P_3$, then $P_4$ lies on at most one of these three lines, so we may assume that $P_4\not\in\tau_2,\tau_3$. Hence $P_2,P_3\in\tau_4$ and thus $\tau_4\cap\pi=P_2P_3$ and thus $\tau_4\cap\pi=\tau_3\cap\pi$. As $P_4\in\tau_4\cap\pi$ and $P_4\notin\tau_3$, this is a contradiction. Consequently, Case 1 does not occur.

\underline{Case 2.} The points $P_1,\ldots,P_4$ are collinear and thus span a line $\ell_0$ of $\pi$.

We denote by $m$ the number of chambers of $X$ whose point is on $\ell_0$ and denote these chambers by $c_1=(P_1,g_1,\tau_1),\ldots,c_m=(P_m,g_m,\tau_m)$. This implies $P_i=\ell_0\cap g_i$. Put $n:=|X|-m$ and let $(R_1,f_1,\mu_1),\ldots,(R_n,f_n,\mu_n)$ be the remaining chambers of $X$, that is those with $R_i\notin \ell_0$. Since we assumed that the $X$-points are not collinear, we have $n\ge 1$.

We first find an upper bound for $m$. Put $I=\{1,\dots,m\}$. For $i\in I$, the chambers $(P_i,g_i,\tau_i)$ and $(R_1,f_1,\mu_1)$ are not opposite, so we have $P_i\in\mu_1$ or $g_i\cap f_1\not=\emptyset$ or $R_1\in\tau_i$. As the point $\ell_0\cap\mu_1$ lies in $q^2$ lines that are not contained in $\pi$, the case $P_i=\mu_1$ can happen for at most $q^2$ indices $i\in I$. If $P_i\notin\mu_1$ and $g_i\cap f_1\not=\emptyset$, then $g_i$ is one of the $q^2$ lines connecting a point $\not=R_1$ of $f_1$ with a point not $\ell_0\cap\mu_1$ of $\ell_0$, therefore this can happen for at most $q^2$ indices $i\in I$.
 Finally consider the $i\in I$ with $R_1\in\tau_i$. For these we have $\ell_0\not\subseteq \tau_i$. If $i,j\in I$ with $\ell_0\not\subseteq \tau_i,\tau_j$, we must have that $g_i$ and $g_j$ meet; this is clear when $P_i=P_j$, and if $P_i\not=P_j$, then it follows from the fact that the two chambers are not opposite and that $P_i\notin\tau_j$ and $P_j\not\in\tau_i$. Hence the lines of these chambers mutually meet and they all meet $\ell_0$. It follows that all these lines contain a point of $\ell_0$ or they are all contained in a plane on $\ell_0$. Hence $m\le 3q^2+q$.

Now we find an upper bound for $n$. Since the chambers $(P_i,g_i,\tau_i)$ for $i=1,2,3,4$, are mutually not opposite and their lines are mutually skew, then $P_i\in\tau_j$ or $P_j\in\tau_i$ for all $i,j=1,2,3,4$. As $P_1,\dots,P_4\in\ell_0$, it follows that at least three of the planes $\tau_1,\dots,\tau_4$ contain $\ell_0$, and we may thus assume that $\ell_0$ lies in $\tau_1$, $\tau_2$ and $\tau_3$. Put $J=\{1,\dots,n\}$.

First consider $j\in J$, with $R_j\notin \tau_4$. Then $R_j\notin\tau_1,\dots,\tau_4$. The plane $\mu_j$ meets $\ell_0$ in a point, so at least three of the points $P_1,\ldots,P_4$ are not in $\mu_j$. Since $(R_j,f_j,\mu_j)$ and $(P_i,g_i,\tau_i)$, $i=1,2,3,4$, are not opposite, it follows that $f_j$ meets at least three of the lines $g_1,\ldots,g_4$. There are four ways to choose three of the four lines $g_1,\ldots,g_4$. If we have chosen three lines, they span a regulus. There are $q+1$ lines that meet all three lines, one of those lines is $\ell_0$. This yields that there are at most $4q$ possible chambers $(R_j,f_j,\mu_j)$, $j\in J$, with $R_j\notin \tau_4$.

Now consider $j\in J$, with $R_j\in \tau_4$.  In this case $R_j$ is one of the $q$ points $\not=P_4$ of the line $\tau_4\cap\pi$. The point $\mu_j\cap \ell_0$ can be equal to at most one of the points $P_1,\ldots,P_3$ and therefore $f_j$ has to meet at least two of the lines $g_1,g_2,g_3$. There are three ways to choose two of the lines $g_1$, $g_2$ and $g_3$. If we have chosen two lines and a point $R_j\not=P_4$ on $\tau_4\cap\pi$, the line $f_j$ is determined. Hence at most $3q$ chambers $(R_j,f_j,\mu_j)$, $j\in J$, satisfy $R_j\in \tau_4$.

Hence $n\leq 4q+3q=7q$ and $|X|=m+n\le 3q^2+8q$.
\end{proof}

\begin{cor} \label{C: points in a plane, points collinear, q^2+q}
	Let $M$ be a maximal  independent set of chambers of $\PG(3,q)$. Let $\pi$ and $X$ be as in Lemma \ref{L: points in a plane}. If all the points of the chambers in $X$ are incident with the line $\ell_0\subseteq\pi$, then at most $q^2+q$ chambers in $X$ contain planes that do not contain $\ell_0$.
\end{cor}

\begin{proof}
 	In the proof of Lemma \ref{L: points in a plane} we showed that distinct chambers of $X$, contain distinct lines. Let $(P_1,g_1,\tau_1)$ and $(P_2,g_2,\tau_2)$ be distinct chambers in $X$ and assume that $\ell_0\not\subseteq\tau_1,\tau_2$. This implies $\tau_i\cap \ell_0=P_i$, for $i=1,2$. Since $X$ is a subset of $M$, the two chambers $(P_1,g_1,\tau_1)$ and $(P_2,g_2,\tau_2)$ are not opposite. If $P_1\neq P_2$, this implies that $g_1$ meets $g_2$. Hence any chambers in $X$, whose planes do not contain $\ell_0$, contain lines that are mutually not skew. Since $\ell_0$ is not an $X$-line, the statement follows from Lemma \ref{L: s(q^2+q+1) no s+1 skew lines}.
\end{proof}

\begin{notation}
Let $M$ be a maximal  independent set of chambers of $\PG(3,q)$.

A line $\ell$ will be called a \emph{$\pi$-line} of $M$, if it does not have weight $(q+1)^2$ and if there exists a plane $\pi$ on $\ell$ such that the pair $(\ell,\pi)$ has weight $q+1$.

A line $\ell$ will be called a \emph{$P$-line} of $M$, if it does not have weight $(q+1)^2$ and if there exists a point $P$ on $\ell$ such that the pair $(P,l)$ has weight $q+1$.
\end{notation}
	
Proposition \ref{P: weight of lines}  shows that every line of weight $q+1$ is a $P$-line or a $\pi$-line, and that the lines of weight $2q+1$ are the lines that are both $P$-lines and $\pi$-lines. Conversely, every $\pi$-line, as well as every $P$-line, has weight $q+1$ or $2q+1$. A $\pi$-line of the projective space is a $P$-line in the dual space.
	
\begin{lemma} \label{L: no pi-lines are skew}
Let $M$ be a maximal  independent set of chambers of $\PG(3,q)$. Let $X$ be the set consisting of the $\pi$-lines of $M$ (this includes the lines of weight $2q+1$) and the lines of weight $(q+1)^2$. Then the lines of $X$ mutually meet. In particular there exists a point or a plane that is incident with all lines of $X$. In particular $|X|\le q^2+q+1$.
\end{lemma}
\begin{proof}
In view of Proposition \ref{P: weight of lines} (a), it suffices to show that any two $\pi$-lines meet. Suppose therefore that $h_1$ and $h_2$ are $\pi$-lines, so that they occur in line-plane flags $(h_1,\tau_1)$ and $(h_2,\tau_2)$ of weight $q+1$. If $h_1$ is contained in $\tau_2$ or if $h_2$ is contained in $\tau_2$, then $h_1$ and $h_2$ meet, since they are contained in a common plane. Suppose now that this is not the case. Then $h_1$ has a point $P_1$ that is not contained in $\tau_2$, and $h_2$ has a point $P_2$ that is not contained in $\tau_2$. Since the chambers $(P_1,h_1,\tau_1)$ and $(P_2,h_2,\tau_2)$ are in $M$ and thus not opposite, it follows that $h_1$ meets $h_2$.
\end{proof}

\begin{lemma} \label{L: intersection of planes}
Let $M$ be a maximal  independent set of chambers of $\PG(3,q)$. Suppose that there exists at least $3q+2$ $\pi$-lines $h_1,\ldots,h_n$, and let $(h_1,\tau_1),\ldots,(h_n,\tau_n)$ with $n\ge 3q+2$, be all incident line-plane pairs of weight $q+1$. Then exactly one of the following three cases occurs.
\begin{enumerate}[(1)]
\item \begin{enumerate}
      \item There exists a point $Q$ that lies on all $\pi$-lines.
      \item The subspace $\ell:=\bigcap_{i=1}^{n}\tau_i$ is a line on $Q$ of weight $(q+1)^2$.
      \item For every line $h\not=\ell$ on $Q$, the line-plane pair $(h,\erz{h,\ell})$ has weight $q+1$.
      \item Every chamber $(P,h,\tau)$ of $M$ satisfies $Q\in h$ or $P\in \ell$.
      \end{enumerate}
\item \begin{enumerate}
      \item There exists a plane $\pi$ incident with all $\pi$-lines.
      \item For $i=1,\dots,n$, we have $\tau_i=\pi$.
      \item For every line $h$ of $\pi$, the line-plane pair $(h,\pi)$ has weight $q+1$.
      \item Every chamber $(P,h,\tau)$ of $M$ satisfies $P\in\pi$.
      \end{enumerate}
\item \begin{enumerate}
      \item There exists a plane $\pi$ incident with all $\pi$-lines.
      \item The subspace $Q:=\bigcap_{i=1}^{n}\tau_i$ is a point $Q$ with $Q\notin\pi$.
      \item For every line $h$ of $\pi$, the line-plane pair $(h,\erz{h,Q})$ has weight $q+1$.
      \item Every chamber $(P,h,\tau)$ of $M$ satisfies $h\subseteq\pi$ or $P=Q$.
      \end{enumerate}
\end{enumerate}
\end{lemma}
\begin{proof}
Recall that the $\pi$-lines mutually meet, so they are either contained in a plane or they contain a common point. As there are more than $q+1$ $\pi$-lines, only one of these cases occurs.

\underline{Part (1).} In this part we consider the case in which the $\pi$-lines are all incident with a point $Q$.

First we claim for any two chambers $(P_1,\ell_1,\pi_1)$ and $(P_2,\ell_2,\pi_2)$ of $M$ with $Q\notin \ell_1,\ell_2$, that the lines $QP_1$ and $QP_2$ coincide. To see this, notice that $Q$  lies on $q+1$ lines that meet $\ell_1$, on as many that meet $\ell_2$ and on at least one that meets $\ell_1$ and $\ell_2$. So $Q$ lies on at most $2q+1$ lines that meet $\ell_1$ or $\ell_2$. Hence, if $I$ is the set consisting of the indices $i$ with $1\le i\le n$, for which $h_i$ is skew to $\ell_1$ and $\ell_2$, then $|I|\ge n-(2q+1)\ge q+1$. For $i\in I$, the line $h_i$ is skew to $\ell_1$ and $\ell_2$, so Proposition \ref{P: weight of lines} implies that $\tau_i$ contains $P_1$ and $P_2$. \\
Assume that $Q$, $P_1$ and $P_2$ are not collinear. This implies for $i\in I$ that $\tau_i$ is the span of these three points $Q$, $P_1$, $P_2$. This implies that the lines $h_i$, with $i\in I$ are lines in $\tau_i$ that are incident with $Q$. Since $|I|\geq q+1$, every line in $\tau_i$ that is incident with $Q$, is a line $h_i$ with $i\in I$. In this case the line $QP_1$ is one of the lines $h_i$, but this cannot be since all lines $h_i$ are skew to $\ell_1$, contradiction.\\
This proves the claim $P_2\in P_1Q$. Furthermore we see that $\bigcap_{i\in I}\tau_i$ is a line, namely $P_1Q$.

For any $i$ with $1\le i\le n$, the line $h_i$ does not have weight $(q+1)^2$, so there exists a chamber $(P_i,\ell_i,\pi_i)$ in $M$ with $\ell_i\cap h_i=\emptyset$. Proposition \ref{P: weight of lines} shows that $P_i\in\tau_i$ and hence the line $QP_i$ lies in $\tau_i$. The above claim shows that all lines $QP_i$ are the same line $\ell$. Hence $\ell$ is a line of all planes $\tau_i$, $1\le i\le n$. Since $n>q+1$, we get $\ell =\bigcap_{i=1}^n\tau_i$.

Since $\ell=QP_i$ for all $i=1,\dots,n$, we get (1)(d). Hence $\ell$ meets every $M$-line, so maximality of $M$ or Proposition \ref{P: weight of lines} implies that $\ell$ has weight $(q+1)^2$. If $h$ is any line on $Q$ other than $\ell$, then part (d) implies that every chamber in $M$ is non-opposite to every chamber that contains $(h,\erz{h,\ell})$. Maximality of $M$ implies that $(h,\erz{h,\ell})$ has weight $q+1$.

\underline{Part(2).} In this part we consider the case in which there exists a plane $\pi$ containing all $\pi$-lines and $\tau_i=\pi$ for all $i=1,\dots,n$.

Then 2(a) and 2(b) holds. Consider any $M$-chamber $(P,\ell,\tau)$. If $\ell$ is a line of $\pi$, then $P\in\pi$. If $\ell$ is not a line of $\pi$, then $\ell\cap\pi$ is a point, so there are indices $i$ for which $\ell$ and $h_i$ are skew; in this case Proposition \ref{P: weight of lines} implies that $P\in\pi_i$, and hence we have also $P\in\pi$. This proves 2(d).

From 2(d) we see that every chamber with plane $\pi$ is non-opposite to every chamber of $M$, so maximality of $M$ implies that these chambers lie in $M$, this proves 2(c).

\underline{Part (3).} In this part we consider the case in which there exists a plane $\pi$ containing all $\pi$-lines and an integer $j$ with $1\le j\le n$ and $\tau_j\not=\pi$.

Since $h_j$ does not have weight $(q+1)^2$, there exists an $M$-chamber $(Q,\ell,\tau)$ whose line $\ell$ is skew to $h_j$. Then Proposition \ref{P: weight of lines} implies that $Q\in\tau_j$ and hence $Q\notin\pi$.

We now prove 3(d). For this consider an $M$-chamber $(Q',\ell',\tau')$. We have to show that $\ell'$ is a line of $\pi$ or $Q'=Q$. Assume that this is not true. Then $\ell$, $\ell'$ and $QQ'$ are three (not necessarily distinct) lines that are not contained in $\pi$, so at most $3q+1$ lines of $\pi$ meet one of these lines. As $n\ge 3q+2$, we can thus find an index $i$ with $h_i$ skew to these three lines. Since $h_i$ is skew to $\ell$ and $\ell'$, then Proposition \ref{P: weight of lines} implies that $Q,Q'\in\tau_i$. But then $h_i$ and $QQ'$ are lines of $\tau_i$, which is a contradiction as we assumed these lines to be skew. This completes the proof of 3(d).

As in the previous cases, 3(c) follows from 3(d) and the maximality of $M$. It remains to show 3(b). For this let $1\le i\le n$. As $h_i$ does not have weight $(q+1)^2$, Proposition \ref{P: weight of lines} shows that there exists an $M$-chamber whose line is skew to $h_i$. By 3(d), the point of this chamber is $Q$, and Proposition \ref{P: weight of lines} shows that $Q\in\tau_i$.
\end{proof}

For the convenience, we also state the dual result.

\begin{lemma} \label{L: intersection of planes dual}
Let $M$ be a maximal  independent set of chambers of $\PG(3,q)$. Suppose that there exists at least $3q+2$ $P$-lines $h_1,\ldots,h_n$, and let $(P_1,h_1),\ldots,(P_n,h_n)$ with $n\ge 3q+2$, be all incident point-line pairs of weight $q+1$. Then exactly one of the following three cases occurs.
\begin{enumerate}[(1)]
\item \begin{enumerate}
      \item There exists a plane $\pi'$ that contains all $P$-lines.
      \item The points $P_1,\dots,P_n$ span a line $\ell'$ of $\pi'$ of weight $(q+1)^2$.
      \item For every line $h\not=\ell'$ of $\pi$, the point-line pair $(h\cap \ell',h)$ has weight $q+1$.
      \item Every chamber $(P,h,\tau)$ of $M$ satisfies $h\subseteq\pi'$ or $\ell'\subseteq\tau$.
      \end{enumerate}
\item \begin{enumerate}
      \item There exists a point $Q'$ incident with all $P$-lines.
      \item For $i=1,\dots,n$, we have $P_i=Q'$.
      \item For every line $h$ on $Q'$, the point-line pair $(Q',h)$ has weight $q+1$.
      \item Every chamber $(P,h,\tau)$ of $M$ satisfies $Q'\in\tau$.
      \end{enumerate}
\item \begin{enumerate}
      \item There exists a point $Q'$ incident with all $P$-lines.
      \item The points $P_1,\dots,P_n$ span a plane $\pi'$ with $Q'\notin\pi'$.
      \item For every line $h$ on $Q'$, the point-line pair $(h\cap\pi',h)$ has weight $q+1$.
      \item Every chamber $(P,h,\tau)$ of $M$ satisfies $Q'\in h$ or $\tau=\pi'$.
      \end{enumerate}
\end{enumerate}
\end{lemma}

\section{Examples}\label{examplesection}

In this section we give examples for maximal independent sets of $\Gamma$. We leave it to the reader to verify that all independent sets described in this section are indeed maximal independent sets.

\begin{example}[Maximal independent sets with $q^4+3q^3+4q^2+3q+1$ chambers] \label{E: largest structure}
Let $x$ be a plane or a point, and let $C(x)$ be the set of all chambers of $\PG(3,q)$ whose lines are incident with $x$. Then the lines of any two chambers of $C(x)$ meet, and hence $C(x)$ is independent. It has $(q^2+q+1)(q+1)^2$ elements.
\end{example}

\begin{example}[Maximal independent sets with $3q^3+5q^2+3q+1$ chambers]  \label{E: second largest structure}
For each example we also give its weight distribution $(a,b,c,d)$, which means that the example has $a$ lines of weight $(q+1)^2$, $b$ lines of weight $2q+1$, $c$ lines of weight $q+1$, no line of weight 2, and $d$ lines of weight $1$.
\begin{enumerate}

\item\label{M1} Distribution $(q+1,q^2,0,q^2)$.
Let $(Q,\ell,\pi)$ be a chamber, and let $M_1$ be the set of all chambers $(P,h,\tau)$ that satisfy at least one of the following conditions.
\begin{enumerate}
\item $Q\in h\subseteq\pi$,
\item $Q\in h\not\subseteq\pi$ and ($\tau=\erz{h,\ell}$ or $P=Q$),
\item $Q\notin h\subseteq \pi$, $P=h\cap \ell$ and $\tau=\pi$.
\end{enumerate}

\item\label{M2} Distribution $(q+1,0,2q^2,0)$.
Let $(Q,\pi)$ be an incident point-plane pair and let $M_2$ be the set of all chambers $(P,h,\tau)$ that satisfy at least one of the following conditions.
\begin{enumerate}
\item $Q\in h\subseteq\pi$,
\item $Q\notin h\subseteq\pi$ and $\tau=\pi$,
\item $Q\in h\not\subseteq\pi$ and $P=Q$.
\end{enumerate}

\item\label{M3} Distribution $(q+1,0,2q^2,0)$.
Let $(Q,\ell,\pi)$ be n chamber and let $M_3$ be the set of all chambers $(P,h,\tau)$ that satisfy at least one of the following conditions.
\begin{enumerate}
\item $Q\in h\subseteq\pi$,
\item $Q\in h\not\subseteq\pi$ and $\tau=\erz{h,\ell}$,
\item $Q\notin h\subseteq\pi$ and $P=h\cap \ell$.
\end{enumerate}

\item\label{M4} Distribution $(1,q^2+q,0,q^3+q^2)$.
Let $(Q,\ell)$ be an incident point-line pair, and let $M_4$ be the set of all chambers $(P,h,\tau)$ that satisfy at least one of the following conditions.
\begin{enumerate}
\item $h=\ell$,
\item $Q\in h\not=\ell$ and ($\tau=\erz{h,\ell}$ or $P=Q$),
\item $P=h\cap\ell\not=Q$ and $\tau=\erz{h,\ell}$.
\end{enumerate}

\item\label{M5} Distribution $(1,q,2q^2,q^3)$.
Let $(Q,\ell,\pi)$ be a chamber and let $M_5$ be the set of all chambers $(P,h,\tau)$ that satisfy at least one of the following conditions.
\begin{enumerate}
\item $h=\ell$,
\item $\ell\not=h\subseteq \pi$ and $\tau=\pi$, 
\item $Q\in h\neq \ell$ and $P=Q$,
\item $Q\notin h\not\subseteq \pi$, $h$ meets $\ell$, $P=h\cap\ell$ and $\tau=\erz{h,\ell}$.
\end{enumerate}

\item\label{M6} Distribution $(1,q,2q^2,q^3)$.
Let $(Q,\ell,\pi)$ be a chamber and let $M_6$ be the set of all chambers $(P,h,\tau)$ that satisfy at least one of the following conditions.
\begin{enumerate}
\item $h=\ell$,
\item $\ell\not=h\subseteq \pi$ and $P=h\cap \ell$,
\item $Q\in h\neq \ell$ and $\tau=\erz{h,\ell}$,
\item $Q\notin h\not\subseteq \pi$, $h$ meets $\ell$, $P=h\cap\ell$ and $\tau=\erz{h,\ell}$.
\end{enumerate}

\item\label{M7} Distribution $(1,0,2(q^2+q),(q-1)(q^2+q))$.
Let $(Q,\ell,\pi)$ be a chamber and let $Q'$ be a point $\neq Q$ on $\ell$. Let $M_7$ be the set of all chambers $(P,h,\tau)$ that satisfy at least one of the following conditions.
\begin{enumerate}
\item $h=\ell$,
\item $Q\in h\not=\ell$ and $\tau=\erz{h,\ell}$,
\item $Q'\in h\not=\ell$ and $P=Q'$,
\item $Q,Q'\notin h$, $h$ meets $\ell$, $P=h\cap\ell$ and $\tau=\erz{h,\ell}$.
\end{enumerate}

\end{enumerate}
\end{example}

\begin{remark} 
The sets $M_2$, $M_3$, $M_5$ and $M_6$ are selfdual. The other sets are not, their dual associates are omitted for comprehensibility.
\end{remark}

\begin{remark}
Two point-plane flags $(P_1,\tau_1)$ and $(P_2,\tau_2)$ are called opposite, if $P_i\notin \tau_j$, for $i,j=1,2$ and $i\neq j$.
	Let $F:=(Q,\ell,\pi)$ be a chamber. Let $C_F$ be the set of point-plane flags $(P,\tau)$ that satisfy one of the following conditions
\begin{enumerate}[(a)] 
	\item $P=Q$,
	\item $\tau=\pi$,
	\item $P\in \ell \subseteq \tau$
\end{enumerate}
It was shown in Theorem 1 of \cite{pointhyperplaneflags} that  the largest maximal sets of pairwise non-opposite point-plane flags are of the form $C_F$.\\
 Let $F$ be a fixed chamber. Let $M$ be the set of all chambers, whose point-plane pair occurs in $C_F$. Then $M$ is $M_5$.
\end{remark}

\begin{example}[Maximal independent set with $3q^3+4q^2+3q+2$ chambers] \label{E: third largest structure}
Let $(Q,\ell,\pi)$ be a chamber and let $Q'$ be a point that is not on $\pi$. Let $M$ be the set of all chambers $(P,h,\tau)$ that satisfy at least one of the following conditions.
\begin{enumerate}[(a)]
\item $Q\in h\subseteq\pi$,
\item $Q\not\in h\subseteq\pi$ and ($\tau=\erz{h,Q'}$ or $P=h\cap \ell$),
\item $Q\in h\not\subseteq\pi$ and $P=Q'$ and $\tau=\erz{h,\ell}$.
\end{enumerate}
The $q+1$ lines incident with $Q$ and $\pi$ have weight $(q+1)^2$, the remaining $q^2$ lines in $\pi$ have weight $2q+1$ (the pairs $(h\cap \ell,h)$ and $(h,\erz{h,Q'})$ each have weight $q+1$), and the line $QQ'$ has weight $1$. Hence $|M|=(q+1)(q+1)^2+q^2(2q+1)+1$.
\end{example}

\section{ The cardinality of maximal independent sets for $q\leq  5$ }

If the order $q$ is small enough, we can compute the size of the second and third largest maximal independent sets. Our code is implemented in GAP4, we used \cite{GAP4}, \cite{FinInG1.5.6} and \cite{GRAPE4.9.0}. The code is available at \cite{GAPcode}.\\

Since the independence number of the Kneser graph on chambers  of $\PG(3,q)$ was already determined in \cite{AlgebraicApproach}, we compute the size $b_2$ of the second largest maximal independent sets and the size $b_3$ of the third largest maximal independent sets. The results for $q\leq 5$ are listed below. For $q>2$ the runtime is substantial.

\begin{center}
\begin{tabular}{|c c c |}
 \hline
$q$ & $b_2$ & $b_3$ \\ [0.5ex]
 \hline\hline
 2 & 51 & 48  \\
 \hline
 3 & 136 & 128 \\
 \hline
 4 & 285 & 270 \\
 \hline
  5 & 516 & 492 \\
 \hline
\end{tabular}
\end{center}

This shows $b_2=3q^3+5q^2+3q+1$ and $b_3=3q^3+4q^2+3q+2$ for $q\leq 5$.

\section{Proof of Theorem \ref{T: main theorem} for $q\geq 4$} \label{Section proof of main theorem}

In this section we prove the remaining parts of Theorem \ref{T: main theorem} for $q\geq 4$  in several propositions, namely \ref{P: weight 2}, \ref{P: many P-lines and pi-lines}, \ref{P: Prop 2q+1 part 1} and \ref{P: Prop 2q+1 second part} - \ref{P: (q^2+q+1)x weight (q+1)^2}. Throughout this section $M$ denotes a maximal independent set of the Kneser graph on chambers  of $\PG(3,q)$. We also assume throughout that $q\geq 4$.

For $q\geq 7$ this section is self-contained, for $q=4,5$ we need the computations from the previous section in the following sense. If a maximal independent set $M$ for $q=4$ or $q=5$ contains less than $q^4+3q^3+4q^2+3q+1$ elements we automatically know that it contains at most $3q^3+5q^2+3q+1$ elements.

 Our first goal is to show that $|M|$ is small enough, in regards to Theorem \ref{T: main theorem}, if all $M$-lines have weight one or two.

\begin{lemma} \label{C: points in a plane q^3}
Consider any plane $\pi$ and one of its lines $\ell$ and let $X$ be the set consisting of all chambers $(P,h,\tau)$ of $M$ for which $h$ has $M$-weight one or two, for which $h$ is skew to $\ell$, and for wich $P=h\cap\pi$. Then $|X|\leq q^3 +16$.
\end{lemma}

\begin{proof}
We may assume $X\not=\emptyset$. Let $(P,h,\tau)$ be a chamber of $X$. Since $h$ is skew to $\ell$, the line $h$ is not incident with $\pi$. Therefore $h\cap \pi$ is a uniquely determined point. Proposition \ref{P: weight of lines} (d) implies that $h$ lies in only one chamber of $X$. This implies that the number of distinct lines of chambers in $X$ is equal to the number of chambers in $X$.\\
 If the points of all chambers of $X$ are collinear, and hence points of a line $g$ of $\pi$, then the point of each chamber of $X$ is one of the $q$ points of $g$ that does not lie on $\ell$. As each such point lies on $q^2$ lines skew to $\ell$, we have $|X|\le q^3$ in this case. Otherwise the points of the chambers of $X$ are non-collinear and Lemma \ref{L: points in a plane} shows $|X|\le 3q^2+8q$ and thus $|X|\leq q^3+16$ as $q\geq 4$.
\end{proof}

\begin{lemma} \label{L: weight 1}
If all $M$-lines have weight one or two, then the number of $M$-lines is at most $3q^3+2q^2+q+33$.
\end{lemma}
\begin{proof}
Let $(Q,\ell,\pi)$ be a chamber in $M$ and consider another chamber $(P,h,\tau)$ of $M$. For each such chamber we have $h\cap\ell\not=\emptyset$, or $P\in\pi$, or $Q\in \tau$. The number of lines $h\neq \ell$ with $h\cap\ell\not=\emptyset$ is at most $(q+1)(q^2+q)$ as each of the $q+1$ points of $\ell$ lies only on $q^2+q$ further lines.

According to Lemma \ref{L: weight 1} the situation $h\cap\ell=\emptyset$ and $P\in \pi$ (and hence $P=h\cap\pi$) occurs  at most $q^3+16$ times. The dual statements shows that the situation $h\cap\ell=\emptyset$ and $Q\in\tau$  occurs  at most $q^3+16$ times. Hence the number of $M$-lines is at most $ 1+(q+1)(q^2+q)+2(q^3+16)$.
\end{proof}

Next we will analyze the situation in which all $M$-lines have weight one or two and there is at least one line of weight two.

\begin{lemma} \label{L: two skew lines with weight 2}
Suppose that $\ell_{1/2}$ and $\ell_{3/4}$ are two skew $M$-lines of weight two. Then every $M$-line intersects $\ell_{1/2}$ or $\ell_{3/4}$.
\end{lemma}
\begin{proof}
Let $(P_1,\ell_{1/2},\pi_1)$, $(P_2,\ell_{1/2},\pi_2)$, $(P_3,\ell_{3/4},\pi_3)$ and $(P_4,\ell_{3/4},\pi_4)$ be the chambers of $\ell_{1/2}$ and $\ell_{3/4}$. From Proposition \ref{P: weight of lines}(d) we know that $P_1\not=P_2$, $\pi_1\not=\pi_2$, $P_3\not=P_4$, and $\pi_3\not=\pi_4$ and it shows also for $i=1,2$ and $j=3,4$ that either $P_i\in\pi_j$ or $P_j\in\pi_i$. As $P_1$ can not lie in $\pi_3$ and $\pi_4$, we may assume that $P_1\notin\pi_3$. Then $P_3\in\pi_1$. Hence $P_4\notin\pi_1$ and thus $P_1\in\pi_4$. Then $P_2\notin \pi_4$ and hence $P_4\in\pi_2$. Then $P_3\notin \pi_2$ and hence $P_2\in\pi_3$. It follows that $\pi_2\cap\pi_3=P_2P_4$ and $\pi_2\cap\pi_4=P_1P_4$.

Let $(Q,h,\tau)$ be a chamber in $M$ and assume that $h$ is skew to $\ell_{1/2}$ and $\ell_{3/4}$.
As we still have symmetry in $1$ and $2$, Lemma \ref{P: weight of lines}(d) shows that we may assume that $P_1\in
\tau$ and $Q\in\pi_2$. Applying the same lemma to $\ell_{3/4}$ and $h$ shows that either $P_1,P_3\in\tau$ and $Q\in\pi_2,\pi_4$ or otherwise $P_1,P_4\in\tau$ and $Q\in\pi_2,\pi_3$.

In the first case we have $Q\in\pi_2\cap \pi_4=P_1P_4$ and since $Q\neq P_1$, we get $Q\notin P_1P_3$. Thus $\tau=\erz{P_1,P_3,Q}=\pi_4$; then $h$ and $\ell_{3/4}$ are lines of $\pi_4=\tau$, a contradiction, since these two lines are skew.

In the second case we have $Q\in\pi_2\cap\pi_3=P_2P_4$, and hence $\tau=\erz{P_1,P_4,Q}=\pi_2$, which now implies that $h$ and $\ell_{1/2}$ meet, again a contradiction.
 \end{proof}

\begin{cor} \label{C: number of lines with weight 2 bound}
The number of lines of $M$-weight two is at most $2(q^2+q+1)$.
\end{cor}

\begin{proof} Lemma \ref{L: two skew lines with weight 2} yields that $M$ does not contain three pairwise skew lines of weight $2$. Using Lemma \ref{L: s(q^2+q+1) no s+1 skew lines}, we get the statement.
\end{proof}

\begin{prop} \label{P: weight 2}
If all $M$-lines have weight one or two, then $|M|\leq 3q^3+3q^2+2q+34$, or $|M|\leq 15q^2+40q+9$.
\end{prop}
\begin{proof}
Let $x$ be the number of $M$-lines and $y$ those of weight two. Then $|M|=x+y$. Lemma \ref{L: weight 1} shows that $x\le 3q^3+2q^2+q+33$. If the lines of weight two mutually meet, then Lemma \ref{L: s(q^2+q+1) no s+1 skew lines} shows that $y\le q^2+q+1$ and we obtain $|M|\leq 3q^3+3q^2+2q+34$. \\
We may therefore assume that there exist two skew lines $\ell_{1/2}$ and $\ell_{3/4}$ of weight two. Let $(P_1,\ell_{1/2},\pi_1)$, $(P_2,\ell_{1/2},\pi_2)$, $(P_3,\ell_{3/4},\pi_3)$ and $(P_4,\ell_{3/4},\pi_4)$ be the chambers of $\ell_{1/2}$ and $\ell_{3/4}$. Lemma \ref{P: weight of lines} shows that $P_i\neq P_j$ and $\pi_i\neq \pi_j$ for $1\le i<j\le 4$.

Lemma \ref{L: two skew lines with weight 2} shows that every $M$-line meets $\ell_{1/2}$ or $\ell_{3/4}$.
There are $(q+1)^2$ lines meeting $\ell_{1/2}$ and $\ell_{3/4}$. Now we estimate the number of chambers in $M$, whose lines do not meet both $\ell_{1/2}$ and $\ell_{3/4}$ and therefore meet exactly one of these lines. From Proposition \ref{P: weight of lines} (d) we know that such a chamber $(Q,h,\tau)$ satisfies one of the following conditions.
\begin{itemize}
	\item The line $h$ meets $\ell_{1/2}$ and the point $Q$ is in $\pi_3$.
	\item The line $h$ meets $\ell_{1/2}$ and the point $Q$ is in $\pi_4$.
	\item The line $h$ meets $\ell_{3/4}$ and the point $Q$ is in $\pi_1$.
	\item The line $h$ meets $\ell_{3/4}$ and the point $Q$ is in $\pi_2$.
\end{itemize}
We find an upper bound $n$ for the number of $M$-lines that belong to chambers of $M$ whose lines meet $\ell_{1/2}$ and whose points are incident with $\pi_3$. Then clearly $4n$ is an upper bound for the number of $M$-lines that do not meet $\ell_{1/2}$ and $\ell_{3/4}$.

The lines $\ell_{1/2}$ and $\ell_{3/4}$ are skew. Therefore $\ell_{1/2}$ and $\pi_3$ intersect in exactly one point. This implies that $q+1$ lines of $\pi_3$ meet $\ell_{1/2}$. Lemma \ref{L: points in a plane} yields that the number of chambers whose point is in $\pi_3$ and whose line is not, is at most $3q^2+8q$, or the points of the chambers (who satisfy this condition) are collinear. In the first case we have $n\le 3q^2+8q+q+1$. Let us now assume that all points are incident with a line $\ell_0\subseteq \pi_3$ and that the lines meet $\ell_{1/2}$. The number of lines who meet $\ell_0$ and $\ell_{1/2}\neq \ell_0$ is at most $2q^2+q+1$; hence $n\le 2q^2+2q+2$ in this case.

Therefore $n\le 3q^2+9q+1$ in any case and thus $x\le 2+(q+1)^2+4n\le 13q^2+38q+7$. Lemma \ref{C: number of lines with weight 2 bound} shows that $y\le 2(q^2+q+1)$, therefore we obtain $|M|\leq 15q^2+40q+9$.
\end{proof}

\begin{lemma} \label{L: 2- and q+1-lines must meet}
An $M$-line of weight two meets every $M$-line of weight larger than two.
\end{lemma}
\begin{proof}
Consider a line $h$ of weight two and a line $\ell$ of weight larger than two. Let $c_1$ and $c_2$ be the two chambers containing $h$. If $\ell$ has weight $(q+1)^2$, then Proposition \ref{P: weight of lines} (a) shows that $h$ and $\ell$ intersect. We may thus assume that $\ell$ has weight $q+1$ or $2q+1$. Assume that $h$ and $\ell$ are skew. Then Proposition \ref{P: weight of lines} (b,c) implies that there exists either a point $P$ on $\ell$ that lies in the planes of $c_1$ and $c_2$, or there exists a plane $\pi$ on $\ell$ that contains the points of $c_1$ and $c_2$. Since $h$ and $\ell$ are skew, it follows that $c_1$ and $c_2$ have the same point or the same plane. This contradicts Proposition \ref{P: weight of lines} (d).
\end{proof}

\begin{lemma} \label{L: 3q+2 pi-lines: no 2-line}
If $M$ has the property that there exist at least $3q+2$ $P$-lines or at least $3q+2$ $\pi$-lines, then there does not exist an $M$-line of weight $2$.
\end{lemma}
\begin{proof} By duality we assume that there exist at least $3q+2$ $\pi$-lines. Then one of the cases in Lemma \ref{L: intersection of planes} occurs. In case (1), we see that every line that meets all $\pi$-lines must pass through $Q$ and then (1)(c) shows that such a line has weight at least $q+1$. In cases (2) and (3), wee see that every line that meets all $\pi$-lines is contained in $\pi$ and part (c) in these cases shows that such a line has weight at least $q+1$. Therefore Lemma \ref{L: 2- and q+1-lines must meet} shows that there does not exist a line of weight 2.
\end{proof}

\begin{prop}\label{P: many P-lines and pi-lines}
Suppose that $M$ contains at least $3q+2$ $\pi$-lines and at least $3q+2$ $P$-lines. Then $|M|\le 2q^3+7q^2+11q+1$ or $M$ is as one of the examples constructed in Example \ref{E: second largest structure} or Example \ref{E: third largest structure}.
\end{prop}
\begin{proof}
From Proposition \ref{P: weight of lines} we know that every $M$-line has weight $1$, $q+1$, $2q+1$ or $(q+1)^2$.
From Lemma \ref{L: intersection of planes} we know that there is a point or plane $x$, called $Q$ or $\pi$ in Lemma \ref{L: intersection of planes}, such that all $\pi$-lines are incident with $x$. Since there are more than $q+1$ $\pi$-lines and since every line of weight $(q+1)^2$ meets every $M$-line, then $x$ is also incident with all lines of weight $(q+1)^2$. Notice that a line of weight $2q+1$ is a $\pi$-line and a $P$-line and hence also incident with $x$. It follows therefore from the different parts (c) of Lemma \ref{L: intersection of planes} that the lines that are incident with $x$ are precisely the lines of weight $(q+1)^2$ and the $\pi$-lines. Dually, there is a point or plane $x'$, called $Q'$ or $\pi'$ in Lemma \ref{L: intersection of planes dual}, such the lines that are incident with $x'$ are precisely the lines of weight $(q+1)^2$ and the $P$-lines. An $M$-line has thus weight one if and only if it is not incident with $x$ and $x'$. And a line has weight $2q+1$ or $(q+1)^2$ if and only if it is incident with $x$ and $x'$. We will use these facts throughout the proof.

For distinct point-line pairs of weight $q+1$, the corresponding lines $h$ are distinct, since otherwise $h$ would have weight at least $2(q+1)$, hence $(q+1)^2$ and thus would not be a $P$-line. Dually, for distinct line-plane pairs of weight $q+1$, the corresponding lines $h$ are distinct.

By Lemma \ref{L: number of lines with weight $(q+1)^2$} the number of $M$-lines of weight $(q+1)^2$ is $0$, $1$, $q+1$ or $q^2+q+1$. However, $q^2+q+1$ is not possible, since then all lines incident with $x$ have weight $(q+1)^2$, contradicting the fact that there are $\pi$-lines. Hence, if $b_i$ is the number of lines of weight $i$, then $b_{q+1}=2(q^2+q+1-b_{2q+1}-b_{(q+1)^2})$ and
\begin{align}\label{eq bound many P-lines and pi-lines}
|M|=\sum_{i}b_ii=b_1+2(q^2+q+1)(q+1)-b_{2q+1}+b_{(q+1)^2}(q^2-1).
\end{align}
There are nine possibilities $(i,j)$ with $1\le i,j\le 3$, where $(i,j)$ means that the $\pi$-lines satisfy (i) of
Lemma \ref{L: intersection of planes} and the $P$-lines satisfy (j) of Lemma \ref{L: intersection of planes dual}. Since the cases $(i,j)$ and $(j,i)$ are dual, we may restrict ourselves to the six cases where $i\le j$. We handle each case separately, where we use the notation given in Lemma \ref{L: intersection of planes} and Lemma \ref{L: intersection of planes dual} in the respective cases, that is the planes $\pi,\pi'$ and the points $Q,Q'$ and the lines $\ell,\ell'$ have the meanings described in these lemmas.
\begin{itemize}
\item[$(1,1)$] In this case the $\pi$-lines are incident with a point $Q$ and Lemma \ref{L: intersection of planes} (1b) states that
 the corresponding special planes meet in a line $\ell$ of weight $(q+1)^2$ on $Q$. Also the $P$-lines are incident with a plane $\pi'$  and their special points span a line $\ell'$ of $\pi'$ of weight $(q+1)^2$. Also for every chamber $(P,h,\tau)\in M$ with a line $h$ of weight one we have $P\in\ell$ and $\ell'\in\tau$.

    Since there are at least $3q+2$ $P$-lines, at least $2q+1$ of these will not contain $Q$. For such a line $h$, there is a point $P$ on $h$ such that $(P,h)$ has weight $q+1$ and by Lemma \ref{L: intersection of planes}(1d) the point $P$ lies on $\ell$. In view of Lemma \ref{L: intersection of planes dual}(1a) every point lies on at most $q+1$ $P$-lines, so Lemma \ref{L: intersection of planes dual}(1b) shows that $\ell=\ell'$. Therefore we have $Q\in\ell=\ell'\subseteq\pi'$. The lines of weight $2q+1$ and $(q+1)^2$ are the lines that are incident with $Q$ and $\pi'$, and so there are $q+1$ of these.

If there is more than one line of weight $(q+1)^2$, then all lines incident with $Q$ and $\pi'$ have weight $(q+1)^2$ and since these lines meet every $M$-line, then every $M$-line is incident with $Q$ or $\pi'$. In this case, there is no line of weight $2q+1$ or $1$, hence
\eqref{eq bound many P-lines and pi-lines} gives $|M|=3q^3+5q^2+3q+1$ and $M$ is as in Example \ref{E: second largest structure}.\ref{M3}.

Now consider the situation when $\ell$ is the only line of weight $(q+1)^2$. Then the other $q$ lines incident with $Q$ and $\pi'$ have weight $2q+1$, so $b_{2q+1}=q$. Lemma \ref{L: intersection of planes} (1c) yields that the $q^2$ lines incident with $Q$, but not $\pi'$, have weight $q+1$. Dually the lines incident with $\pi'$, but not $Q$, also have weight $q+1$, hence $b_{q+1}=2q^2$. Therefore an $M$-line of weight $1$, is not incident with $Q$ or $\pi'$. Lemma \ref{L: intersection of planes} (1d) and Lemma \ref{L: intersection of planes dual} (1c) imply that a chamber $(P,h,\tau)$, with a line that is not incident with $Q$ or $\pi'$, satisfies $P\in\ell$ and $\ell \subseteq \tau$. Every one of the $q$ points $\neq Q$ on $\ell$ is incident with $q^2$ lines that are not in $\pi'$, hence $b_1=q^3$ and \eqref{eq bound many P-lines and pi-lines} gives $|M|=3q^3+5q^2+3q+1$. We get that $M$ is as in Example \ref{E: second largest structure}.\ref{M6}

\item[(1,2)] In this case the $\pi$-lines are incident with a point $Q$ and Lemma \ref{L: intersection of planes} (1b) states that the corresponding special planes meet in a line $\ell$ of weight $(q+1)^2$ on $Q$. The $P$-lines are incident with a point $Q'$ and $Q'$ is their special point. Also for every chamber $(P,h,\tau)\in M$ with a line $h$ of weight one we have $P\in\ell$ and $Q'\in\tau$. We distinguish two cases.

   First we consider the situation when $Q\not=Q'$. As $Q'$ is the special point of every $P$-line, then Lemma \ref{L: intersection of planes} (1d) shows that $Q'\in\ell$. Then $\ell$ is the only line of weight $(q+1)^2$, so $b_{(q+1)^2}=1$ and moreover $b_{2q+1}=0$ and $b_{q+1}=2(q^2+q)$. If $(P,h,\tau)\in M$ with a line $h$ of weight one, then $P$ is a point of $\ell$ other than $Q$ and $Q'$ and the plane $\tau$ contains $Q'$ and hence $\ell\subseteq\tau$. Maximality of $M$ implies that $(h\cap l,h,\erz{h,\ell})\in M$ for every line $h$ that meets $\ell$ in a point other than $Q$ and $Q'$. Hence $b_1=(q-1)(q^2+q)$ and \eqref{eq bound many P-lines and pi-lines} gives $|M|=3q^3+5q^2+3q+1$ and $M$ is as in Example \ref{E: second largest structure}.\ref{M7}.

    Now we consider the situation when $Q=Q'$. Then every line on $Q$ has weight $2q+1$ or $(q+1)^2$. If $(P,h,\tau)\in M$ with a line $h$ of weight one, then $P$ is a point of $\ell$ other than $Q$ and the plane $\tau$ contains $Q'$ and hence $\ell\subseteq\tau$. If $\ell$ is the only line of weight $(q+1)^2$, it follows that $M$ is as in Example \ref{E: second largest structure}.\ref{M4}. If there are $q+1$ lines of weight $(q+1)^2$, they span a plane $\tau_0$ on $\ell$ and every line of weight one must lie in $\tau_0$. In this case $b_{2q+1}=q^2$, $b_{(q+1)^2}=q+1$ and $b_1=q^2$, hence \eqref{eq bound many P-lines and pi-lines} gives $|M|=3q^3+5q^2+3q+1$ and $M$ is as in Example \ref{E: second largest structure}.\ref{M1}.

\item[(1,3)] In this case the $\pi$-lines are incident with a point $Q$ and Lemma \ref{L: intersection of planes} (1b) states that the corresponding special planes meet in a line $\ell$ of weight $(q+1)^2$ on $Q$. The $P$-lines are incident with a point $Q'$ and their special points span a plane $\pi'$ with $Q'\notin\pi'$. Also for every chamber $(P,h,\tau)\in M$ with a line $h$ of weight one we have $P\in\ell$ and $\tau=\pi'$.

    We fist show that $Q=Q'$. Assume on the contrary that $Q\not=Q'$. By Lemma \ref{L: intersection of planes}(1d), the point $Q'$ lies on $\ell$. Consider a chamber $(P,h,\tau)$ of $M$ with a $P$-line $h$. Then $h\not=\ell$ and thus $Q\notin h$. Then Lemma \ref{L: intersection of planes}(1d) shows that $P\in \ell$, so $P=Q'$. But Lemma \ref{L: intersection of planes dual} shows that $P\in\pi'$ and since $Q'\notin\pi'$ this is a contradiction.

    Therefore $Q=Q'$. Then every line on $Q$ has weight $2q+1$ or $(q+1)^2$. Then $P_0:=\pi'\cap\ell$ is a point and for all $(P,h,\tau)\in M$ with a line $h$ of weight one we have $P=P_0$ and $\tau=\pi'$, which implies that $P_0\in h\subseteq\pi'$. Hence $b_1\le q+1$. If $\ell$ is not the only line of weight $(q+1)^2$ and $\ell'$ is a second one, then every $M$-line meets $\ell$ and $\ell'$, so every line of weight one meets $\ell'$ and is thus the line joining $P_0$ and $\ell'\cap\pi'$. In this case $b_{2q+1}=q^2$, $b_{(q+1)^2}=q+1$ and $b_1=1$, so \eqref{eq bound many P-lines and pi-lines}  gives $|M|=3q^3+4q^2+3q+2$ and $M$ is as in Example \ref{E: third largest structure}. If $\ell$ is the only line of weight $(q+1)^2$, then $b_{2q+1}=q^2+q$ and $b_{(q+1)^2}=1$ and $b_1\le q+1$, so \eqref{eq bound many P-lines and pi-lines}  gives $|M|\leq 2q^3+4q^2+4q+2$.

\item[(2,2)] In this case there exists a plane $\pi$ incident with all $\pi$-lines and a point $Q'$ incident with $P$-lines. Also for every $\pi$-line the corresponding plane is $\pi$, and for every $P$-line, the corresponding point is $Q$. Finally, if $(P,h,\tau)\in M$, then $P\in\pi$ and $Q\in\tau$. Let $(P_i,h_i,\tau_i)$, $i=1,\dots,b_1$ be the elements of $M$ for which $h_i$ is a $1$-line.

Lemma \ref{L: intersection of planes} (2d) shows that every point of every $M$-chamber has to be in $\pi$. Therefore $Q'\in\pi$. Then the $q+1$ lines incident with $Q'$ and $\pi$ have weight $2q+1$ or $(q+1)^2$, and the $2q^2$  lines that are incident with exactly one of the subspaces $Q'$ and $\pi$, have weight $q+1$, that is $b_{q+1}=2q^2$.

    If there is more than one line of weight $(q+1)^2$, then all lines incident with $Q'$ and $\pi$ have weight $(q+1)^2$. Since a line of weight $(q+1)^2$ meets every $M$-line, it follows in this case that every $M$-line is incident with $Q'$ or $\pi$, so that $b_1=b_{2q+1}=0$, $b_{(q+1)^2}=q+1$ and $b_{q+1}=2q^2$. In this case, \eqref{eq bound many P-lines and pi-lines} gives $|M|=3q^3+5q^2+3q+1$ and hence $M$ is as in Example \ref{E: second largest structure}.\ref{M2}.

    If there is exactly one line $\ell$ of weight $(q+1)^2$, then $b_{2q+1}=q$. Hence, if $(P,h,\tau)\in M$ with a line $h$ of weight one, then Lemma \ref{L: intersection of planes}(2d) shows $P=h\cap\pi\in\ell$ and $Q'\in\tau$ and thus $\ell\subseteq\tau$. It follows for every line $h$ that meets $\ell$ but is not incident with $Q'$ or $\pi$ that the chamber $(h\cap\ell,h,\erz{h,\ell})$ is non-opposite to any chamber of $M$. Hence $b_1=q^3$ so \eqref{eq bound many P-lines and pi-lines} gives $|M|=3q^3+5q^2+3q+1$ and $M$ is as in Example \ref{E: second largest structure}.\ref{M5}.

    Finally we assume that there is no line of weight $(q+1)^2$. Then $b_{2q+1}=q+1$ and $b_{(q+1)^2}=0$, so \eqref{eq bound many P-lines and pi-lines} gives $|M|\le 2q^2(q+1)+(q+1)(2q+1)+b_1$. We use Lemma \ref{L: points in a plane}, if the lines of weight one meet $\pi$ in non-collinear points, then we have $b_1\le 3q^2+8q$ and hence $|M|\le 2q^3+7q^2+11q+1$ and we are done.

    So assume that all lines $h$ of weight one meet $\pi$ in a point of a line $\ell\subseteq\pi$. Then $Q$ is not a point of $\ell$, since otherwise $\ell$ would meet every $M$-line and thus have weight $(q+1)^2$, a contradiction since there are no lines of weight $(q+1)^2$ in the present situation. Since $Q\notin\ell$, it follows as before that the lines $h_i$ mutually meet and since they all meet $\ell$, they are all incident with a point on $\ell$ or lie all in a plane on $\ell$. In any case $b_1\le q^2+q<3q^2+8q$ and we are done.

\item[(2,3)] In this case there exists a plane $\pi$ incident with all $\pi$-lines and a point $Q'$ incident with $P$-lines. Also $\pi$ is the special plane of each $\pi$-line. Furthermore the special points of the $P$-lines span a plane $\pi'$ and the point $Q'$ does not lie on $\pi'$. Finally, if $(P,h,\tau)\in M$, then $P\in\pi$ and moreover we have $Q'\in h$ or $\tau=\pi'$.

Lemma \ref{L: intersection of planes} (2d) states that every point of every $M$-chamber in $\pi$. The plane $\pi'$ is the span of points of $M$-chambers, hence $\pi=\pi'$.

    Since $Q'\notin\pi'$, no line is incident with $Q'$ and $\pi$, which implies that $b_{q+1}=2(q^2+q+1)$ and $b_{2q+1}=b_{(q+1)^2}=0$. There can be no chamber $(P,h,\tau)$ in $M$ with a line of weight one, since otherwise $\tau=\pi'=\pi$, but the lines of $\pi$ have weight $q+1$. Hence $|M|=b_{q+1}(q+1)=2(q^2+q+1)(q+1)$ and we are done.

\item[(3,3)] Lemma \ref{L: intersection of planes} (3) shows that there exits a point $Q$ and a plane $\pi$ such that every chamber $(P,h,\tau)$ of $M$ satisfies $h\subseteq\pi$ or $P=Q$. Lemma \ref{L: intersection of planes dual} (3) shows that there exits a point $Q'$ incident with all $P$-lines. Furthermore there is a plane $\pi'$ that is not on $Q'$, so that every chamber $(P,h,\tau)$ of $M$ satisfies $Q'\in h$, or $\tau=\pi'$.

First let us assume that $Q'\in \pi$. In this case at most $q+1$ $P$-lines are in $\pi$. Therefore at least $2q+1$ distinct $P$-lines are not in $\pi$. The special point of these lines has to be $Q$, therefore they are all equal to $QQ'$, a contradiction.

Now let us assume that $Q'\notin \pi$. In this case the special point of every $P$-line has to be $Q$, but these points span a plane $\pi'$, contradiction. It follows that $(3,3)$ does not occur at all.
\end{itemize}
This completest the proof.
 \end{proof}

Now we look at maximal independent sets that do not contain $3q+2$ $\pi$-lines and $3q+2$ $P$-lines.

\begin{lemma} \label{L: pi-lines and P-lines give bound for number of chambers}
If $n$ is the number of $\pi$-lines and $m$ is the number of $P$-lines, then the number of chambers in $M$ that have a line of weight $q+1$ or $2q+1$ is $(n+m)(q+1)-s$ where $s$ is the number of lines that are $\pi$-lines and $P$-lines at the same time.
\end{lemma}
\begin{proof}
If a line is a $\pi$-line or a $P$-line, but not both, then it has weight $q+1$. If a line is a $\pi$-line and a $P$-line, then it has weight $2q+1$. Hence, if $s$ lines are $\pi$-lines and $P$-lines at the same time, then the number in question is $(m-s)(q+1)+(n-s)(q+1)+s(2q+1)$.
\end{proof}

\begin{prop} \label{P: Prop 2q+1 part 1}
	Let $M$ contain at least $3q+2$ $\pi$-lines, or at least $3q+2$ $P$-lines, but no line of weight $(q+1)^2$. Then $|M|\leq 2q^3+7q^2+11q+1$.
\end{prop}
\begin{proof}
By duality we may assume that there exist at least $3q+2$ $\pi$-lines in $M$. Lemma \ref{L: 3q+2 pi-lines: no 2-line} implies that lines of weight $2$ do not exist. Since there does not exist a line of weight $(q+1)^2$, we have that every line has weight $0$, $1$, $q+1$ or $2q+1$, that is every $M$-line is a $\pi$-line or a $P$-line or has weight $1$. Hence, if $m$ is the number of $\pi$-lines and $n$ is the number of $P$-lines and $x$ is the number of lines of weight one, then the previous lemma implies that
\begin{align}\label{eq Prop 2q+1 part 1}
|M|\le (m+n)(q+1)+x-s,
\end{align}
where $s$ is the number of lines that are $P$-lines and $\pi$-lines at the same time. If $n\ge 3q+2$, then Proposition \ref{P: many P-lines and pi-lines} along with the observation that all examples in \ref{E: second largest structure} or \ref{E: third largest structure} posses lines of weight $(q+1)^2$, shows that $|M|\leq 2q^3+7q^2+11q+1$. From now on we assume $n\le 3q+1$.

Since no line has weight $(q+1)^2$ and since there are at least $3q+2$ $\pi$-lines, we see that Case (2) or Case (3) of Lemma \ref{L: intersection of planes} must occur. This implies that there exists a plane $\pi$ such that the lines of $\pi$ are exactly the $\pi$-lines. In particular there are exactly $q^2+q+1$ $\pi$-lines, that is $m=q^2+q+1$.

First consider the situation when Case (3) of Lemma \ref{L: intersection of planes} occurs. Then there exists a point $Q$ that does not lie in $\pi$, such that every chamber of $M$ has its line in $\pi$ or its point is $Q$. Since the lines of $\pi$ have weight at least $q+1$, it follows that every line of weight one contains $Q$. Hence $x\le q^2+q+1$. Since $n\le 3q+1$ and $m=q^2+q+1$, then \eqref{eq Prop 2q+1 part 1} proves the claim.

Now we consider the situation, in which Case (2) of Lemma \ref{L: intersection of planes} occurs, that is $(h,\pi)$ has weight $q+1$ for all lines $h$ of $\pi$, and every chamber of $M$ has its point in $\pi$. Therefore every chamber $(P,\ell,\pi)$ with a line of weight $1$ satisfies $\ell\not\subseteq\pi$ and $P=\ell\cap\pi$. Lemma \ref{L: points in a plane} shows $x\leq (q+1)q^2$ or $x\leq 3q^2+8q$. Since $q\geq 4$  \eqref{eq Prop 2q+1 part 1} proves the claim in both cases.
\end{proof}

 Now we consider the case in which there are at most $3q+1$ $\pi$-lines and at most $3q+1$ $P$-lines in $M$ and also at least one $\pi$- or $P$-line. In the next Lemma we study the case in which $M$ contains a $\pi$-line. The case in which $M$ contains a $P$-line is dual. We want to find an upper bound on the number of $M$-lines with weight $1$ in this case.

\begin{lemma} \label{L: bound on 1-lines, if one line has weight q+1}
Suppose that $M$ contains a $\pi$-line. Then there are at most $4q^2+16+q(q+1)^2$ lines with $M$-weight $1$.
\end{lemma}

\begin{proof}
Let $h$ be a $\pi$-line and let $\tau$ be the plane for which the pair $(h,\tau)$ has weight $q+1$. 
Denote by $X$ the set consisting of chambers of $M$ whose lines have weight one and are skew to $h$, and denote by $Y$ the set of chambers of $M$ whose lines have weight one and meet $h$. We have to show that
\begin{align}\label{eqn_X+Y}
|X|+|Y|\le 4q^2+16+q(q+1)^2.
\end{align}
Clearly the number of chambers in $X$ is equal to the number of distinct lines in chambers of $X$.
The line $h$ meets at most $q(q+1)^2$ lines and hence $|Y|\le q(q+1)^2$. We may therefore assume that $|X|>4q^2+16$. As $q\geq 4$, this implies $|X|>3q^2+8q$. If $(P,g,\pi)\in X$, then Proposition \ref{P: weight of lines} (c)(i) implies that $P\in\tau$ and hence $P=g\cap\tau$.
Lemma \ref{L: points in a plane} shows that there exists a line $\ell$ in $\tau$ such that $P\in\ell$ for all $(P,g,\pi)\in X$. As $X\not=\emptyset$, this implies that $\ell\not=h$, so $h$ and $\ell$ meet in a point $L$. Since $\ell$ meets $q^3$ lines that do not meet $h$, then $|X|\le q^3$. Since we have to verify \eqref{eqn_X+Y}, it suffices to show that $|Y|\le 6q^2+q+16$. Assume that this is not correct. We shall derive a contradiction.

Since $|Y|> 6q^2+q+16>2q^2+q$, there exists a chamber $(Q,f,\mu)$ in $Y$ for which the point $f\cap h$ is distinct from $\ell$ and for which $f$ is not contained in $\tau$. Lemma \ref{L: points in a plane} shows that $Q\notin\tau$, since $Q\notin \ell$.

We define subsets $X_1,\dots,X_4$ of $X$ as follows. A chamber $(P,g,\pi)\in X$ belongs to $X_1$ iff $\ell\not\subseteq \pi$, it belongs to $X_2$ iff $\ell\subseteq\pi$ and $\mu\cap\ell\in g$, it belongs to $X_3$, iff $\ell\subseteq\pi$ and $\mu\cap \ell\notin g$ and $Q\in\pi$, and it belongs to $X_4$ iff $\ell\subseteq\pi$ and $\mu\cap\ell\notin g$ and $Q\notin\pi$. Clearly $X$ is the union of these four sets.

Consider two chambers $(P_i,g_i,\pi_i)\in X_1$, $i=1,2$. If $P_1=P_2$, then $g_1$ and $g_2$ share the point $P_1$. If $P_1\not=P_2$, then $P_1\notin \pi_2$ and $P_2\notin \pi_1$ (since $P_1,P_2\in \ell\not\subseteq \pi_i$) and hence $g_1\cap g_2\not=\emptyset$, as the two chambers belong to $M$ and are therefore not opposite. We have shown that the lines of the chambers of $X_1$ mutually meet. Since all these lines meet also $\ell$, it follows that the chambers of $X_1$ contains at most $q^2+q$ distinct lines. Hence $|X_1|\le q^2+q$.

A chamber $(P,g,\pi)\in X_2$ satisfies and $P=\mu\cap\ell$, and hence $g$ is one of the $q^2$ lines on $\mu\cap\ell$ that does not lie in $\tau$. Therefore $|X_2|\le q^2$.

A chamber $(P,g,\pi)\in X_3$ satisfies $\pi=\erz{\ell,Q}$, and hence $g$ is one of the $(q-1)q$ lines of this plane that contain neither $h\cap\ell$ nor $\mu\cap\ell$. Therefore $|X_3|\le (q-1)q$.

A chamber $(P,g,\pi)\in X_4$ satisfies $Q\notin\pi$ and $P\notin \mu$. Since $(P,g,\pi)$ and $(Q,f,\mu)$ are in $M$ and hence not opposite, it follows that $g$ and $f$ meet. Hence $g$ is one of the lines that joins one of the $q-1$ points $\not=h\cap\ell,\mu\cap\ell$ of $\ell$ with one of the $q-1$ points $\not=Q,f\cap\tau$ of $f$. Hence $|X_4|\le (q-1)^2$. If follows that
\begin{align*}
|X|\le|X_1|+|X_2+|X_3|+|X_4|=q^2+q+q^2+(q-1)q+(q-1)^2=4q^2-2q+1.
\end{align*}
But $|X|> 4q^2+16$, a contradiction.
\end{proof}

\begin{prop} \label{P: Prop 2q+1 second part}
	Suppose that $M$ contains at most $3q+1$ $\pi$-lines and at most $3q+1$ $P$-lines, but at least one $\pi$ or $P$-line. Furthermore suppose that $M$ contains no line of weight $(q+1)^2$. Then $|M|\leq q^3+16q^2+13q+22$.
\end{prop}

\begin{proof}Without loss of generality we assume that $M$ contains a $\pi$-line. There are at most $2(3q+1)$ lines in $M$ that are $\pi$-lines or $P$-lines. Lemma \ref{L: pi-lines and P-lines give bound for number of chambers} yields that the number of chambers in $M$ that contain a line of weight $q+1$, or $2q+1$, is at most $2(3q+1)(q+1)$. Corollary \ref{C: number of lines with weight 2 bound} yields that at most $2(q^2+q+1)$ lines in $M$ have weight $2$. Finally Lemma \ref{L: bound on 1-lines, if one line has weight q+1} yields that at most $(4q^2+16)+q(q+1)^2$ lines in $M$ have weight $1$. Therefore we obtain $|M|\leq 2(3q+1)(q+1)+2(q^2+q+1)2+(4q^2+16)+q(q+1)^2=q^3+16q^2+13q+22$.
\end{proof}

Now we treat the case in which at least one $M$-line has weight $(q+1)^2$ and in which $M$ does not contain $3q+2$  $\pi$- and $P$-lines. The case in which $M$ does contain $3q+2$  $\pi$- and $P$-lines is dealt with in Proposition \ref{P: many P-lines and pi-lines}. Lemma \ref{L: number of lines with weight $(q+1)^2$} states that $M$ contains $0$, $1$, $q+1$, or $q^2+q+1$ lines of weigh $(q+1)^2$. This allows us to split up the case into the following propositions.

\begin{prop} \label{P: 1x weight (q+1)^2}
If there is exactly one $M$-line $\ell$ of weight $(q+1)^2$ and at most $3q+1$ $\pi$-lines or at most $3q+1$ $P$-lines then 
$|M|\leq 2q^3+9q^2+6q+3$.
\end{prop}
\begin{proof}
Let $w$ be the number of lines that have weight $(q+1)^2$, that is $w=1$.
Let $x$ be the number $P$-lines plus the number of $\pi$-lines, let $y$ be the number of lines of weight 2 and let $z$ be the number of lines of weight 1. Lemma \ref{L: pi-lines and P-lines give bound for number of chambers} shows that at most $x(q+1)$ chambers of $M$ contain a $P$-line or a $\pi$-line.  Therefore
\begin{align}\label{eqn_onebigline}
|M|\le w(q+1)^2+x(q+1)+y\cdot 2+z\cdot 1.
\end{align}
Proposition \ref{P: weight of lines} shows that $\ell$ meets all $M$-lines. Hence there can be at most $1+(q+1)(q^2+q)$ $M$-lines that is $x+y+z\le (q+1)(q^2+q)$.
Lemma \ref{L: no pi-lines are skew} implies that there exist at most $q^2+q$ $\pi$-lines and at most $q^2+q$ $P$-lines. However our condition implies that not both cases can occur and we have $x\le (q^2+q)+(3q+1)=q^2+4q+1$. Corollary \ref{C: number of lines with weight 2 bound} shows $y\le 2(q^2+q+1)$. Finally we have $z\leq (q+1)(q^2+q)-x-y$
Therefore \eqref{eqn_onebigline} proves the claim.
\end{proof}

\begin{prop} \label{P: (q+1)x weight (q+1)^2}
If there exist exactly $q+1$ $M$-lines with weight $(q+1)^2$ and at most $3q+1$ $\pi$-lines or at most $3q+1$ $P$-lines, we have $|M|\leq 2q^3+10q^2+6q+3$.
\end{prop}

\begin{proof} Lemma \ref{L: number of lines with weight $(q+1)^2$} (b) implies that all lines of weight $(q+1)^2$ are incident with a point $P$ and a plane $\pi$. In this case the remaining $q^2$ lines incident with $\pi$ (but not with $P$) and the remaining $q^2$ lines incident with $P$ (but not with $\pi$) are the only lines that can be in $M$, since these are the only lines that meet all lines of weight $(q+1)^2$.

Let $w$ be the number of lines that have weight $(q+1)^2$, that is $w=q+1$.
Let $x$ be the number $P$-lines plus the number of $\pi$-lines, let $y$ be the number of lines of weight 2 and let $z$ be the number of lines of weight 1. Lemma \ref{L: pi-lines and P-lines give bound for number of chambers} shows that at most $x(q+1)$ chambers of $M$ contain a $P$-line or a $\pi$-line and we get that \eqref{eqn_onebigline} holds also in this case. We have already seen that $x+y+z\le 2q^2$.

Lemma \ref{L: no pi-lines are skew} implies that there exist at most $q^2$ $\pi$-lines and at most $q^2$ $P$-lines. However our condition implies that not both cases can occur and we have $x\le q^2+(3q+1)=q^2+3q+1$. Corollary \ref{C: number of lines with weight 2 bound} shows $y\le 2(q^2+q+1)$. Finally we have $z\leq 2q^2-x-y$.
Therefore \eqref{eqn_onebigline} proves the claim.
\end{proof}

\begin{prop} \label{P: (q^2+q+1)x weight (q+1)^2}
If there are $q^2+q+1$ $M$-lines with weight $(q+1)^2$, we have $|M|=q^4+3q^3+4q^2+3q+1$ and $M$ is one of the independent sets described in Example \ref{E: largest structure}.
\end{prop}

\begin{proof}If there are $q^2+q+1$ $M$-lines with weight $(q+1)^2$, all these lines are incident with a point $P$ or a plane $\pi$, since they are mutually not skew. Let $x$ be the point or plane that is incident with all lines of weight $(q+1)^2$. Since the lines of weight $(q+1)^2$ meet all $M$-lines, all $M$-lines have to be incident with $x$. This implies that all $M$-lines have weight $(q+1)^2$ and are incident with $x$. Hence $|M|=(q^2+q+1)(q+1)^2=q^4+3q^3+4q^2+3q+1$.
\end{proof}

This completes the proof of Theorem \ref{T: main theorem} for $q\geq 4$.

\section{The chromatic number} \label{C: chromatic number}

Let $\Gamma$ be the Kneser graph on chambers  of $\PG(3,q)$ and let $V$ be the vertex set of $\Gamma$. It was shown in Example 17 in \cite{Thechromaticnumberoftwo} that the chromatic number of $\Gamma$ is at most $q^2+q$. Now we will find a lower bound and therefore determine the chromatic number of $\Gamma$. We need the structure of the largest maximal independent sets for our proof.

\begin{prop} \label{P: chromatic number for q>3}
	For $q\geq 4$, the chromatic number of the Kneser graph on chambers  of $\PG(3,q)$ is $q^2+q$.
\end{prop}

The following proof is adopted from Theorem 18 in \cite{Thechromaticnumberoftwo}. We replace the bound $45 q^3 + 45 q^2 + 9 q + 9$ by our bound $3q^3+5q^2+3q+1$.

\begin{proof}Let $\chi$ be the chromatic number of $\Gamma$. Then there exist independent sets $G_1, \ldots, G_{\chi}$ that partition the vertex set $V$. We know that $\chi \leq q^2+q$ holds. For $1\leq i\leq \chi$ let $F_i$ be a maximal independent set of $\Gamma$ with $G_i\subseteq F_i$. We define $b_1:=q^4+3q^3+4q^2+3q+1$ and $b_2:=3q^3+5q^2+3q+1$. Theorem \ref{T: main theorem} implies that $|F_i|=b_1$ or $|F_i|\leq b_2$ for all $1\leq i\leq \chi$.

Let $I$ be the set of indices $i$ with $|F_i|=b_1$. Theorem \ref{T: main theorem} implies that for $i\in I$ the set $F_i$ is as in Example \ref{E: largest structure}. For $i\in I$ this means that there is a point or a plane $x_i$ and $F_i=C(x_i)$ holds, i.e. $F_i$ is the set of all chambers whose lines are incident with $x_i$. We define $A:=\{ x_i \mid i\in I, x_i \text{\ is a point} \}$ and $B:=\{ x_i \mid i\in I, x_i \text{\ is a plane} \}$.  In Theorem 2.2 of \cite{Howmanyssubspacesmustmissapointset} it is shown that the number of lines of $\PG(3,q)$ that contain a point of $A$, is at most $(q^2+q+1)+|A|q^2$. Dually the number of lines of $\PG(3,q)$ that are contained in a plane of $B$, is at most $(q^2+q+1)+|B|q^2$. Since the line of every chamber of $C(x)$ with $x\in A$ contains $x$, it follows that at most $((q^2+q+1)+|A|q^2)(q+1)^2$ chambers are contained in $C(x)$ for some $x\in A$ and dually at most $((q^2+q+1)+|B|q^2)(q+1)^2$ chambers are contained in $C(x)$ for some $x\in B$. We define $\theta:=(q^2+q+1)$. Since the total number of chambers in $\PG(3,q)$ is $(q^2+1)\theta (q+1)^2$ and since every chamber lies in at least one set $F_i$, we obtain
\begin{align*}
\begin{split}
    (q^2+1)\theta (q+1)^2\leq \left| \bigcup\limits_{i=1}^{\chi} F_{i}\right| \leq (2\theta +|A|q^2+|B|q^2)(q+1)^2 +(\chi -|A|-|B|)b_2.
    \end{split}
\end{align*}
Considering that $q\geq 4$, we have $b_2=3q^3+5q^2+3q+1<q^2(q+1)^2$. Since $|A|+|B|=|I|\leq\chi$, we receive the estimate
\begin{align*}
    (q^2+1)\theta (q+1)^2\leq (2 \theta +\chi q^2)(q+1)^2.
\end{align*}
This implies $\chi > q^2+q-1$ and hence $\chi =q^2+q$.
\end{proof}

 For $q=2,3$, we can calculate the chromatic number $\chi$ using the code from \cite{GAPcode}. In this case we also find $\chi=q^2+q$. Combining this computations with Proposition \ref{P: chromatic number for q>3} yields Theorem \ref{T: chromatic number}.

\subsection*{Acknowledgement}
The authors would like to thank Thomas Titz Mite for many helpful discussions and help with the GAP-code.

\bibliographystyle{plain}

\end{document}